\documentclass{article}

\usepackage{amsfonts,amsmath,amssymb,amsthm}
\usepackage{amscd}
\usepackage[dvipsnames,table,xcdraw]{xcolor}
\usepackage[colorlinks = true,
            linkcolor = Fuchsia,
            urlcolor  = ForestGreen,
            citecolor = WildStrawberry,
            anchorcolor = blue]{hyperref}

\usepackage[all,2cell]{xy} \UseAllTwocells
\usepackage{tikz}
\usetikzlibrary{cd}
\usetikzlibrary{arrows,automata}
\usetikzlibrary{decorations.markings}
\usetikzlibrary{decorations.pathmorphing,shapes}

\usepackage{enumitem}

\usepackage[english]{babel}

\usepackage[letterpaper,top=2cm,bottom=2cm,left=3cm,right=3cm,marginparwidth=1.75cm]{geometry}

\usepackage[style=numeric, backend=bibtex,url=false,maxnames=10]{biblatex}
\usepackage{csquotes}
\makeatletter
\AtEveryBibitem{%
  \ifnameundef{author}{}{%
    \iffieldequals{fullhash}{\bbx@lasthash}
      {\renewcommand{\printnames}[1]{\bibnamedash}}
      {\savefield{fullhash}{\bbx@lasthash}}%
  }%
}
\makeatother

\usepackage{amsmath}
\usepackage{graphicx}

\def\lra{\longrightarrow}
\def\Ext{\mathsf{Ext}}
\def\mf{\mathfrak}

\def\Q{\mathbb Q}
\def\Z{\mathbb Z}

\newcommand{\id}{\mathsf{id}}

\newcommand{\Coh}{\mathsf{Coh}}

\newcommand{\one}{\mathbf{1}}
\newcommand{\kk}{\mathbf{k}}

\newcommand{\mcO}{\mathcal{O}}
\newcommand{\mcL}{\mathcal{L}}

\newcommand{\mcS}{\mathcal{S}}
\newcommand{\mcF}{\mathcal{F}}
\newcommand{\mcA}{\mathcal{A}}
\newcommand{\mcG}{\mathcal{G}}

\newcommand{\mbs}{\mathbf{s}}
\newcommand{\wtR}{\widetilde{R}}
\newcommand{\wtK}{\widetilde{K}}
\newcommand{\PP}{\mathbb{P}}
\newcommand{\UzeroA}{\mathcal{U}^0_{\mcA}}

\DeclareMathOperator{\ch}{ch}

\newcommand{\bbinom}[2]{
  \begin{bmatrix}
    #1 \\
    #2
  \end{bmatrix}
}

\newcommand{\bmat}[3]{
  \begin{bmatrix}
    #1 ; #2 \\
    #3
  \end{bmatrix}
}

\newcommand{\oplusop}[1]{{\mathop{\oplus}\limits_{#1}}}

\theoremstyle{definition}
\newtheorem{thm}{Theorem}[section]

\newtheorem{remark}[thm]{Remark}
\newtheorem{cor}[thm]{Corollary}
\newtheorem{prop}[thm]{Proposition}
\newtheorem{example}[thm]{Example}

\theoremstyle{definition}

\title{Coherent sheaves on projective spaces and a lifting of the integral form of the Cartan subalgebra for quantum $\mf{sl}(2)$}
\author{Mikhail Khovanov}
\date{August 10, 2026}

\begin{document}
\maketitle

\begin{abstract}
We categorify a lifting of the integral form of the Cartan subalgebra for quantum $\mathfrak{sl}(2)$, via categories of equivariant coherent sheaves on projective spaces.
\end{abstract}

\tableofcontents

\section{Introduction}

The main result of the present paper identifies a graded algebra lifting Lusztig's integral Cartan algebra for quantum $\mathfrak{sl}(2)$ with the Grothendieck ring of a symmetric monoidal category of $\mathbb{G}_m$-equivariant coherent sheaves on projective spaces. Under this identification, the quantum-binomial multiplication formula of Guti\'errez-Mart\'inez-Szwej-Wildon~\cite{GMSW1} is lifted to  explicit direct sum decompositions of equivariant vector bundles.

\vspace{0.07in}

Consider a commutative $\Z[q,q^{-1}]$-algebra $R$ with generators $K_{c,n}$ over all $n\ge 0$ and $c\in \Z$ modulo defining relations \eqref{eq_q_binom} and \eqref{eq_multiplication},  see Section~\ref{sec_integral}. It is graded, with $\deg K_{c,n}=n$ , and $R=\oplus_{n\ge 0}R_n$,  the decomposition into homogeneous summands. Algebra $R$ surjects onto the integral form $\UzeroA=\UzeroA(\mf{sl}_2)$ of the Cartan subalgebra of quantum $\mf{sl}_2$ introduced in~\cite{Lusztig2,Lusztig1}, see also~\cite{MartinezRuiz2025,GMSW1}.  The surjective ring homomorphism
\begin{equation}\label{eq_hom_surj}
\psi:R\lra \UzeroA
\end{equation}
is given by
\[
K_{c,1}\longmapsto [K;c] = \frac{q^c K - q^{-c} K^{-1}}{q-q^{-1}},
\]
and, more generally,
\begin{equation}\label{eq_Kcn_int}
K_{c,n} \longmapsto
\bmat{K}{c}{n} := \frac{[K;c][K;c-1]\dots [K;c-n+1]}{[n]!}.
\end{equation}
Our main theorem is a realization of $R$ as the sum of $\mathbb{G}_m$-equivariant Grothendieck groups of $\PP^n$, over all $n$:
\begin{thm}\label{thm_isom}
    There is a natural isomorphism of graded $\Z[q,q^{-1}]$-algebras
    \begin{equation}\label{eq_thm_isom}
    R\ \cong \oplusop{n\ge 0}\, K_0^{\mathbb{G}_m}(\PP^n)
    \end{equation}
    that takes $K_{c,n}$ to the symbol of the equivariant line bundle $\mcO(-c)$ on $\PP^n$.
\end{thm}
Under this isomorphism, $R_n\cong K_0^{\mathbb{G}_m}(\PP^n)$, and the multiplication
\[
R_n\times R_m\lra R_{n+m}
\]
comes from the map on Grothendieck groups induced by the equivariant pushforward functor $(f_{n,m})_{\ast}$ for the quotient map
\begin{equation}\label{eq_fnm}
f_{n,m}:\PP^n\times \PP^m\lra \PP^{n+m}.
\end{equation}
This functor is exact on equivariant coherent sheaves, and it turns the direct sum of categories of equivariant coherent sheaves
\begin{equation}\label{eq_coh_all_n}
\Coh^{\mathbb{G}_m}(\PP) := \oplusop{n\ge 0}\Coh^{\mathbb{G}_m}(\PP^n)
\end{equation}
into an abelian symmetric monoidal category, with the Grothendieck ring of $\Coh^{\mathbb{G}_m}(\PP)$ given by the sum in the RHS of \eqref{eq_thm_isom}.

 $K_0^{\mathbb{G}_m}(\PP^n)$, the Grothendieck group of $\mathbb{G}_m$-equivariant vector bundles (or coherent sheaves) on $\PP^n$, is a free $\Z[q,q^{-1}]$-module with a basis of symbols $[\mcO(-i)]$, $0\le i\le n$ of line bundles.
 We consider the standard $\mathbb{G}_m$-action lifted to $\mcO(1)$ such that its space of global sections $H^0(\PP^n, \mcO(1))$ has character $[n+1]$. This extends to $\mcO(i)=\mcO(1)^{\otimes i}$, where the space of global sections $H^0(\PP^n, \mcO(i))$ has character given by the quantum binomial coefficient $\left[ \begin{smallmatrix} n+i \\ i \end{smallmatrix} \right]$ for $i\ge 0$.

 Computational part of the theorem is given by Theorem~\ref{thm_pushforward_equiv},  which determines the equivariant pushforward of the external tensor product of line bundles on $\PP^n$ and $\PP^m$, subject to a suitable restriction on their degrees. The proof is based on the Horrocks' Criterion. Relations \eqref{eq_multiplication} lift to a decomposition of this pushforward into a direct sum of line bundles on $\PP^{n+m}$. Relations \eqref{eq_q_binom} correspond to the standard long exact sequences of multiples of line bundles $\mcO(-c-k)$ on $\PP^n$. Section~\ref{sec_equiv_coherent}  is devoted to the proof of Theorem~\ref{thm_isom}.

\vspace{0.07in}

Specializing $q$ to $1$ reduces the ground ring to $\Z$ and $R$ to a ring $\wtR$ defined in Section~\ref{sec_integral} with a spanning set $\{\wtK_{c,n}\}$. It has defining relations \eqref{eq_q_binom} and \eqref{eq_multiplication} but with the quantum binomials replaced by the usual binomials, see also~\eqref{eq_multiplication_cl}.

\begin{thm}\label{thm_isom_noneq}
    There is a natural ring isomorphism
\begin{equation}\label{eq_thm_isom_noneq}
    \wtR\ \cong \oplusop{n\ge 0}\, K_0(\PP^n)
    \end{equation}
    that takes $\wtK_{c,n}$ to the symbol of the line bundle $\mcO_{\PP^n}(-c)$ for all $c\in \Z$, $n\ge 0$.
\end{thm}

This theorem is proved in Section~\ref{subsec_pushforward} (see Proposition~\ref{prop_iso_noneq} there). One replaces $\mathbb{G}_m$-equivariant coherent sheaves on $\PP^n$ by the usual coherent sheaves, without any equivariance condition. In the paper, we first treat the non-equivariant case, in Section~\ref{sec_coherent}, then move onto the $\mathbb{G}_m$-equivariant case, in Section~\ref{sec_equiv_coherent}.

\vspace{0.07in}

 Map $\psi$ in \eqref{eq_hom_surj} has a nontrivial kernel. In particular, unlike $R$, algebra $\UzeroA$ does not have a natural grading by the degree $n$ of generators in \eqref{eq_Kcn_int}. Element
 \begin{equation}\label{eq_kernel_el}
 K_{c+2,n}-(q^n+q^{-n})K_{c+1,n}+K_{c,n}-K_{c,n-2}
 \end{equation}
 is nonzero in $R$ but its image in $\UzeroA$ is $0$, see \cite[Proposition 6.1]{GMSW1}. It is an open question, raised in~\cite{MartinezRuiz2025} and in the  papers~\cite{GMSW1,GMSW2,GMSW3}, whether the algebra $\UzeroA$ or the entire Lusztig integral form $\mathcal{U}_{\mcA}(\mf{sl}_2)$ of quantum $\mf{sl}_2$ admits a categorification (generators and relations for these forms can be found in~\cite{Lusztig2}). In the known categorifications of quantum groups~\cite{Lauda,KhLau,Rouquier1,Rouquier2}  the Cartan subalgebra is absent from the Grothendieck rings, having been replaced by a system of idempotents $\{1_{\lambda}\}$, over all  $\lambda$'s in the integral root lattice~\cite{LusztigBook}. The present paper  contains only a categorification of the larger ring $R$ surjecting onto $\UzeroA$ via the map \eqref{eq_hom_surj}.

\vspace{0.07in}

 {\bf Acknowledgments.}
 Proofs of the main technical results of this paper: Theorem~\ref{thm_pushforward} and Theorem~\ref{thm_pushforward_equiv}  were obtained in close discussions with ChatGPT 5.6 Sol and Gemini 3.1 Pro. The author was partially supported by the Simons Collaboration Award 994328 ``New Structures in Low-Dimensional Topology''.

%
%

\section{Ring \texorpdfstring{$R$}{R} and  the integral Cartan subalgebra of quantum \texorpdfstring{$\mathfrak{sl}(2)$}{sl(2)}}
\label{sec_integral}

{\bf Ring $R$.}
Start with the ring $\Z[q,q^{-1}]$
of Laurent polynomials in $q$ with integer coefficients. Recall the definion of (balanced) quantum integers, factorials and binomials:
\begin{equation}\label{eq_balanced}
[n]! :=[n][n-1]\dots [1], \ \ [n]=\frac{q^n-q^{-n}}{q-q^{-1}}\in \Z[q,q^{-1}], \ \
\bbinom{n}{k}:=\frac{[n]!}{[k]![n-k]!},
\end{equation}
and also set
$\begin{bmatrix} n \\ k \end{bmatrix}=0$ for $k<0$ or $k>n$.

Define $R$ to be a commutative $\Z[q,q^{-1}]$-algebra
with generators $K_{c,n}$ over all $n\ge 0$ and $c\in \Z$ and the following relations \eqref{eq_q_binom} and \eqref{eq_multiplication}:
\begin{equation}\label{eq_q_binom}
    \sum_{k=0}^{n+1} (-1)^k \bbinom{n+1}{k} K_{c+k,n} =0
\end{equation}
for any $n\ge 0$ and any $c\in \Z$;
\begin{equation}\label{eq_multiplication}
K_{c,n}K_{b,m} = \sum_{k \ge 0} \begin{bmatrix} n - c + b \\ k - c \end{bmatrix} \begin{bmatrix} m - b + c \\ k - b \end{bmatrix}
K_{k,n+m},
\end{equation}
for $n,m\ge 0$ and $0\le c\le n$, $0\le b\le m$. The sum above is over $k$ such that $\max(b,c) \le k \le  \min(n+b, m+c)$.
The upper arguments of the $q$-binomials in \eqref{eq_multiplication} are non-negative, due to our restrictions on the indices $c,n,b,m$. We also have $K_{c,0}=1$ for any $c\in \Z$.

\begin{example}\label{example_motivate}
Specializing $n=1,2$ gives us
\begin{eqnarray*}
K_{c+2,1}- [2]K_{c+1,1}+ K_{c,1} & = & 0 \\
K_{c+3,2}-[3]K_{c+2,2}+[3]K_{c+1,2} - K_{c,2} & = & 0 .
\end{eqnarray*}
\end{example}

Relations \eqref{eq_q_binom} and \eqref{eq_multiplication} are homogeneous in $n$. Setting $\deg K_{c,n}=n$ for all $c$ and $\deg(q)=0$, we see that $R$ is an $\Z_+$-graded algebra. Denote the degree $n$ term by $R_n$, so that $R=\oplus_{n\ge 0}R_n$.

Note that
relation \eqref{eq_q_binom} at ``level'' $n$ has $n+2$ terms, and it allows to express $K_{d,n}$ for any $d$ as a $\Z[q,q^{-1}]$-linear combination of $n+1$ consecutive ones $K_{c,n},K_{c+1,n},\dots, K_{c+n,n}$. Consequently, $R$ is a $\Z[q,q^{-1}]$-linear span of elements $K_{c,n}$ over all $0\le c\le n$.

Let $\mcS=\{(c,n)|0\le c\le n\}$.
We will write $K_{\mbs}, \mbs\in\mcS$ to denote generators from this set, so that $R$ is a $\Z[q,q^{-1}]$-span of $K_{\mbs}$, $\mbs\in \mcS$.
In the relation~\eqref{eq_multiplication} the generators are limited to $K_{\mbs}$ for $\mbs\in \mcS$.

\begin{thm}\label{thm_R_free_mod}
    $R$ is a free $\Z[q,q^{-1}]$-module with a basis $\{K_{\mbs}\}_{\mbs\in \mcS}$.
\end{thm}

\begin{proof}
 This theorem can be proved by converting \eqref{eq_multiplication} into its module version: form  a free $\Z[q,q^{-1}]$-module $V$ with a basis $\{v_{\mbs}\}_{\mbs\in \mcS}$ and define an action of operators $\widehat{K}_{c,n}$ on it, for $(c,n)\in \mcS$, by the rule
\begin{equation}\label{eq_multiplication_mod}
\widehat{K}_{c,n}v_{b,m} = \sum_{i \ge 0} \begin{bmatrix} n - c + b \\ i - c \end{bmatrix} \begin{bmatrix} m - b + c \\ i - b \end{bmatrix}
v_{i,n+m}.
\end{equation}
Checking that operators $\widehat{K}_{c,n}$ commute and relation \eqref{eq_multiplication} holds (with $K$'s replaced by $\widehat{K}$'s) is a  tedious computation that involves applying the ${}_4\phi_3$ transformation formula (Sears' identity) for $q$-binomials, or related identities ($q$-Chu-Vandermonde). We omit the details.

Alternatively, the theorem follows from the identification of $R$ via Theorem~\ref{thm_isom}
and an explicit computation of the pushforward of the external tensor product of suitable equivariant line bundles in  Theorem~\ref{thm_pushforward_equiv}.
\end{proof}

Consequently, $R_n$ is a free $\Z[q,q^{-1}]$-module with a basis $(K_{0,n},K_{1,n},\dots, K_{n,n})$.
Here are the products of degree 1 generators:
    \[
    K_{0,1}^2 = K_{0,2} + K_{1,2}, \ \ K_{0,1}K_{1,1} = [2] K_{1,2}, \ \ K_{1,1}^2 = K_{1,2} + K_{2,2}.
    \]
\begin{remark}
    It is sometimes convenient to reindex and define
    \[K(a,b):=K_{a,a+b}.\]
    The set $\mcS$ of indices providing the above basis of $R$ corresponds to pairs $(a,b)$ with $a,b\ge 0$. Relation \eqref{eq_multiplication} becomes
    $$K(a,b)K(a',b') = \sum_{i \ge 0} \begin{bmatrix} a' + b \\ i - a \end{bmatrix} \begin{bmatrix} a + b' \\ i - a' \end{bmatrix} K(i, a + a' + b + b' - i), \ \ $$
    the sum over $i$ with
    $\max(a,a')\le i\le a+a'+\min(b,b')$.
\end{remark}

\vspace{0.07in}

{\bf Ring $\wtR$.}
Consider the base ring homomorphism $\Z[q,q^{-1}]\lra \Z$ taking $q$ to $1$. We use it to form the ring
\begin{equation}
    \wtR:= \Z\otimes_{\Z[q,q^{-1}]} R.
\end{equation}
Denote $\wtK_{c,n}:=1\otimes K_{c,n}$. Ring $\wtR$ is a free abelian group with a basis $\{\wtK_{c,n}\}$ over all $0\le c\le n$. Multiplication in this basis is given by equation \eqref{eq_multiplication} with quantum binomials replaced by classical binomials:
\begin{equation}\label{eq_multiplication_cl}
\wtK_{c,n}\wtK_{b,m} = \sum_{k \ge 0} \begin{pmatrix} n - c + b \\ k - c \end{pmatrix} \begin{pmatrix} m - b + c \\ k - b \end{pmatrix}
\wtK_{k,n+m}.
\end{equation}
Ring $\wtR$ is naturally graded, $\wtR=\oplus_{n\ge 0}\wtR_n$, with $\deg \wtK_{c,n}=n$ for all $c$.
There is a surjective homomorphism from $\wtR$ onto the ring of integer-valued polynomials in a variable $x$ via
\[
 \wtK_{c,n}\longmapsto \binom{x+c}{n}.
\]

{\bf Lusztig integral form of the Cartan subalgebra for quantum $\mf{sl}(2)$.}
Consider the ring $\Q[q^{\pm 1},K^{\pm 1}]$ of Laurent polynomials in two commuting variables $q,K$ and localize it by inverting $q^n-q^{-n}$ for all $n>0$. Denote the resulting ring by $R'$.
Let
\begin{equation}\label{eq_Kc1}
[K;c] := \frac{q^c K - q^{-c} K^{-1}}{q-q^{-1}} \in R'
\end{equation}
and define
\begin{equation}\label{eq_Kcn}
\bmat{K}{c}{n} := \frac{[K;c][K;c-1]\dots [K;c-n+1]}{[n]!}\in R'.
\end{equation}
Here $c\in \Z$ is an integer and $n\ge 0$.
We have
\[
\bmat{K}{c}{0}=1, \ \forall c\in\Z, \ \mathrm{and} \ \mathrm{define} \  \bbinom{K}{n}:= \bmat{K}{0}{n}.
\]
Lusztig integral form $\UzeroA$ of the quantum Cartan subalgebra for $\mf{sl}_2$ is defined as the $\Z[q,q^{-1}]$-subalgebra of $R'$ generated by $\bmat{K}{c}{n}$ over all $n\ge 0$, $c\in \Z$. This algebra contains the elements
\begin{equation}\label{eq_K_and_inverse}
K = \bmat{K}{1}{1} - q^{-1}\bbinom{K}{1} \  \mathrm{and} \ K^{-1}= \bmat{K}{1}{1} - q\bbinom{K}{1}.
\end{equation}
A basis of $\UzeroA$ is given, for example, by
\begin{equation}\label{eq_basis_U}
\Bigl\{ \bbinom{K}{n}: n\ge 0 \Bigr\} \cup \Bigl\{ \bmat{K}{1}{n}: n\ge 1 \Bigr\},
\end{equation}
see~\cite[Proposition~5.3]{GMSW1}.

\begin{prop}\label{prop_binom}
    The following relations hold in $R'$ and $\UzeroA$:
\begin{equation}\label{eq_q_binom_2}
    \sum_{k=0}^{n+1} (-1)^k \bbinom{n+1}{k} \bmat{K}{c+k}{n} =0
\end{equation}
\end{prop}
This proposition must be well-known to researchers in quantum groups. A proof can be found in Section~\ref{subsec_prop_binom}.

\begin{prop}
For $n,m\ge 0$ and $0\le c\le n$, $0\le b\le m$
    the following relations hold:

\begin{equation}\label{eq_multiplication_2}
\bmat{K}{c}{n} \bmat{K}{b}{m} = \sum_{k \ge 0} \begin{bmatrix} n - c + b \\ k - c \end{bmatrix} \begin{bmatrix} m - b + c \\ k - b \end{bmatrix}
\bmat{K}{k}{n+m}.
\end{equation}
\end{prop}
The sum above is over $k$ such that $\max(b,c) \le k \le  \min(n+b, m+c)$, due to $\begin{bmatrix} n \\ k \end{bmatrix}=0$ for $k<0$ or $k>n$. See~\cite[Theorem 1.2]{GMSW2} or~\cite{MartinezRuiz2025}
for a proof.
$\square$

We see that there is a surjective homomorphism $\psi:R\lra \UzeroA$ in \eqref{eq_hom_surj} given by \eqref{eq_Kcn_int}. Comparing the bases of these two $\Z[q,q^{-1}]$-algebras shows that $\psi$ is not injective. A nonzero element in the kernel is shown in \eqref{eq_kernel_el}.

%
%

\section{Coherent sheaves on projective spaces}
\label{sec_coherent}

\subsection{Coherent sheaves on symmetric powers of a curve}

Fix an algebraically closed field $\kk$ and consider a smooth curve $Y$ over $\kk$, not necessarily compact. Let $Y_n:=S^n(Y)$, be the $n$-th symmetric power of $Y$,
which parametrizes effective divisors of degree $n$ on $Y$. It is a smooth $n$-dimensional variety, and there are natural ``addition'' maps
\begin{equation}
    Y_n \times Y_m \stackrel{f_{n,m}}{\lra} Y_{n+m}, \qquad
(D,E)\longmapsto D+E.
\end{equation}
Map $f_{n,m}$ sends a pair of effective divisors to their sum.

\begin{prop}\label{prop_good_morphism}
    Morphism $f_{n,m}$ is finite, flat and surjective. It is finite locally free of degree $\deg(f_{n,m})=\binom{n+m}{n}.$
\end{prop}

\begin{proof}
Restricting to an affine patch $U=\mathrm{Spec}\, A\subset Y$, the map $f_{n,m}:S^n(U)\times S^m(U)\lra S^{n+m}(U)$ is induced by the inclusion of commutative rings
\[
(A^{\otimes (n+m)})^{S_{n+m}} \subset (A^{\otimes n})^{S_n} \otimes (A^{\otimes m})^{S_m},
\]
where $S_n\times S_m\subset S_{n+m}$ is a maximal parabolic subgroup.
Over $Y$, $f_{n,m}$ is the morphism of finite quotients induced by the inclusion
\[
S_n\times S_m
\subset
S_{n+m}\colon
\qquad
Y^{n+m}/(S_n\times S_m)
\longrightarrow
Y^{n+m}/ S_{n+m}.
\]
It follows that $f_{n,m}$ is finite. It is surjective because every effective divisor of degree $n+m$ is a sum of effective subdivisors of degrees $n$ and $m$. Both the source and the target are smooth varieties of dimension $n+m$. In particular, the source is Cohen--Macaulay and the target is regular. Miracle flatness, see~\cite[Theorem 18.16(b)]{Eisenbud} or~\cite[Tag 00R4]{stacks-project},  implies that the finite morphism $f_{n,m}$ is flat. Since the morphism is both finite and flat, it is finite locally free. The degree of $f_{n,m}$ can be computed over the open subset of $Y_{n+m}$ consisting of reduced divisors. A divisor supported at $n+m$ distinct points has exactly $
\binom{n+m}{n}$
effective subdivisors of degree $n$.
\end{proof}

Also note that $f_{n,m}$ is proper and affine.
The addition maps are commutative and associative. Precisely,
$f_{n,m}=f_{m,n}\circ\tau,$
where
\[
\tau\colon Y_n\times Y_m\longrightarrow Y_m\times Y_n
\]
interchanges the two factors, and
\[
f_{n+m,\ell}\circ
\bigl(f_{n,m}\times\operatorname{id}_{Y_\ell}\bigr)
=
f_{n,m+\ell}\circ
\bigl(\operatorname{id}_{Y_n}\times f_{m,\ell}\bigr).
\]
Consider the pushforward and pullback functors between abelian categories of coherent sheaves on these varieties induced by the map $f_{n,m}$:
\begin{equation}\label{eq_maps_fnm}
\begin{tikzcd}[cramped, sep=large]
\Coh(Y_n\times Y_m)
  \arrow[r, shift left=0.7ex, "{(f_{n,m})_*}"]
& \Coh(Y_{n+m})
  \arrow[l, shift left=0.7ex, "{(f_{n,m})^*}"]
\end{tikzcd}
\end{equation}

\begin{prop}
    Functors $(f_{n,m})_{\ast} $ and $(f_{n,m})^{\ast}$ are exact.
\end{prop}
\begin{proof}
For a finite morphism $f$, the functor $f_{\ast}$ is exact on the categories of coherent sheaves. For a flat morphism $f$, the functor $f^{\ast}$ is exact.
\end{proof}

Map $f_{n,m}$ induces a biexact bifunctor
\[
(f_{n,m})_{\circ} \ : \
\Coh(Y_n)\times \Coh(Y_m) \lra \mathsf{Coh}(Y_n\times Y_m) \stackrel{(f_{n,m})_{\ast}}{\lra}
\Coh(Y_{n+m}),
\]
given by
\[
(\mcF_1,\mcF_2)\longmapsto (f_{n,m})_{\ast}(\mcF_1\boxtimes \mcF_2).
\]
The category
\[
\Coh_{\ast}(Y) \ := \ \oplusop{n\ge 0} \Coh(Y_n)
\]
is abelian. Summing $(f_{n,m})_{\circ}$ over all $n,m\ge 0$ gives a biexact bifunctor
\[
f_{\circ} : \Coh_{\ast}(Y) \times \Coh_{\ast}(Y) \lra \Coh_{\ast}(Y)
\]
making $\Coh_{\ast}(Y)$ into a symmetric monoidal abelian category. The unit object $\one$ is given by the structure sheaf of the point $Y_0$. Let us also write the bifunctor in this category as
\begin{equation}\label{eq_tensor_prod}
\mcF_1 \widetilde{\boxtimes} \mcF_2 := (f_{n,m})_{\circ}(\mcF_1,\mcF_2) = (f_{n,m})_{\ast}(\mcF_1\boxtimes \mcF_2).
\end{equation}

\subsection{Pushforward for the symmetrizing map  from \texorpdfstring{$\PP^n\times \PP^m$}{P(n)*P(m)} to \texorpdfstring{$\PP^{n+m}$}{P(n+m)}}
\label{subsec_pushforward}

Let us now specialize from an arbitrary smooth curve $Y$ to $Y=\PP^1$. The symmetric powers $S^n(\PP^1)=\PP^n$. Denote the map $f_{n,m}:\PP^n\times \PP^m\lra \PP^{n+m}$  simply by $f$.
 To compute the pushforward $f_{\ast}$ of the sheaf $\mcO_{\PP^n}(i) \boxtimes \mcO_{\PP^m}(j)$ on $\PP^n\times \PP^m$, for $i,j$ in a suitable range of $j-i$, we use the Horrocks' Criterion, see~\cite[Theorem 2.3.1]{OSS}, and then do a computation to match the Hilbert polynomials.
The multiplicities in Theorem~\ref{thm_pushforward} below  arise from the Chu-Vandermonde convolution identity \eqref{eq_CVa}.

\begin{thm}\label{thm_pushforward} The pushforward sheaf of $\mcO_{\PP^n}(i) \boxtimes \mcO_{\PP^m}(j)$ under the map $f=f_{n,m}$ splits as a direct sum of line bundles on $\PP^{n+m}$ if  $-m \le j - i \le n$. When this condition holds, the pushforward is given by:
\begin{equation}\label{eq_f_decomposition}
f_{\ast}(\mcO_{\PP^n}(i) \boxtimes \mcO_{\PP^m}(j)) \cong \bigoplus_{k=0}^{m} \binom{j-i+m}{k} \binom{i-j+n}{m-k}\mcO_{\PP^{n+m}}(i + k - m).
\end{equation}
\end{thm}

(To make the numbers easier to read, we write $u\mcL$ instead of $\mcL^{\oplus u}$ to denote the direct sum of $u$ copies of $\mcL$.)

\begin{proof}

\noindent The map $f: \PP^n \times \PP^m \to \PP^{n+m}$ corresponds to the tensor product mapping $\operatorname{Sym}^n V \otimes \operatorname{Sym}^m V \to \operatorname{Sym}^{n+m} V$, where $V\cong \kk^2$. Pulling back a hyperplane class from $\PP^{n+m}$ restricts to a bilinear form on the product. Therefore, the pullback of the twisting sheaf is
\begin{equation}\label{eq_pullback_O}
f^{\ast} (\mcO_{\PP^{n+m}}(1)) \cong \mcO_{\PP^n}(1) \boxtimes \mcO_{\PP^m}(1).
\end{equation}

\noindent Let $\mathcal{E} = f_*(\mcO_{\PP^n}(i) \boxtimes \mcO_{\PP^m}(j))$. Note that $\mathcal{E}$ is a bundle, since the pushforward of a locally free sheaf along a finite, flat morphism (Proposition~\ref{prop_good_morphism}) of smooth varieties is a locally free sheaf.
To determine when $\mathcal{E}$ splits into a direct sum of line bundles we use Horrocks' Criterion~\cite[Theorem 2.3.1]{OSS}, which requires the intermediate cohomology to vanish: $H^r(\PP^{n+m}, \mathcal{E}(t)) = 0$ for all $0 < r < n+m$ and all twists $t \in \mathbb{Z}$.
Using relative cohomology and the K\"unneth formula, we have:
$$H^r(\PP^{n+m}, \mathcal{E}(t)) \cong H^r(\PP^n \times \PP^m, \mcO(i+t) \boxtimes \mcO(j+t))\cong  \oplusop{p+q=r} \left( H^p(\PP^n, \mcO(i+t)) \otimes H^q(\PP^m, \mcO(j+t)) \right).$$
Because line bundles on projective spaces only have non-zero cohomology in degrees $0$ and the top dimension, the only intermediate degrees $r$ where this product can be non-zero are $r=n$ and $r=m$:
\begin{itemize}
    \item $H^n \neq 0$ if and only if $i+t \le -n-1$ and $j+t \ge 0$, which implies $j - i \ge n+1$.
    \item $H^m \neq 0$ if and only if $j+t \le -m-1$ and $i+t \ge 0$, which implies $i - j \ge m+1$.
\end{itemize}
Thus, intermediate cohomology vanishes  if the difference between the twists is bounded as follows: $-m \le j - i \le n$. When this holds, Horrocks' Criterion guarantees that $f_{\ast}(\mcO(i) \boxtimes \mcO(j))$ splits into a direct sum of line bundles.

\vspace{0.07in}

\noindent Assuming $-m \le j - i \le n$, we can write $\mathcal{E} \cong \bigoplus_k \mcO_{\PP^{n+m}}(a_k)^{\oplus C_k}$ for a finite sequence $(a_k)$ with $\dots a_k < a_{k+1} < \dots$ and $C_k\ge 1$. We find the twists $a_k$ and multiplicities $C_k$ by comparing the dimensions of global sections for sufficiently large $t$ as above.
On $X = \PP^n \times \PP^m$, for large $t$, the dimension of the space of sections is:
\begin{equation}\label{eq_h0one}
h^0(X, \mcO(i+t) \boxtimes \mcO(j+t)) = \binom{t+i+n}{n}\binom{t+j+m}{m}
\end{equation}
On $\PP^{n+m}$, that dimension for the direct sum is:
\begin{equation}\label{eq_h0two}
h^0\left(\PP^{n+m}, \bigoplus_k \mcO(a_k + t)^{\oplus C_k}\right) = \sum_k C_k \binom{t + a_k + n + m}{n+m}.
\end{equation}
Expressions \eqref{eq_h0one} and \eqref{eq_h0two} are equal for all large integer $t$. Consequently, they are equal as polynomials in $t$. Shift the variable $t$ to  $x = t+i+n$ and define  $d = j - i + m - n$. The splitting condition $-m \le j-i \le n$ is equivalent to $-n \le d \le m$. Recall the Chu-Vandermonde identity:
\begin{equation}\label{eq_CVa}
\binom{x}{n}\binom{x+d}{m} = \sum_{k=0}^m \binom{n+d}{k}\binom{m-d}{m-k}\binom{x+k}{n+m}.
\end{equation}
``Numerators'' of the three binomials in \eqref{eq_h0one} and \eqref{eq_h0two} that contain $t$ match the three binomial ``numerators'' in \eqref{eq_CVa} that contain $x$.
Consequently, numbers $C_k$ are given by the two other binomial products on the RHS of \eqref{eq_CVa}.
Substituting back $x = t+i+n$, the basis terms on the right side are $\binom{t+i+n+k}{n+m}$. Matching the arguments of the binomial coefficients yields:
$$t + i + n + k = t + a_k + n + m \implies a_k = i + k - m.$$
The multiplicities of line bundles are given by:
$$C_k = \binom{n+d}{k}\binom{m-d}{m-k} = \binom{j-i+m}{k}\binom{i-j+n}{m-k}.$$
This completes the proof of Theorem~\ref{thm_pushforward}.
\end{proof}

\begin{remark}
    The range of $k$ in \eqref{eq_f_decomposition} with nonzero coefficients is
$$\max(0, j +m - i - n) \le k \le m + \min(0, j - i)$$
In terms of $d = j - i + m - n$, these conditions can be written as
$\max(0,d)\le k\le \min(m,n+d)$.
\end{remark}

\begin{remark} The abelian category of coherent sheaves on $\PP^k$ satisfies the Krull-Schmidt property of the uniqueness of decomposition of an object into a direct sum of indecomposables (with multiplicities), which for our pushforward sheaves are all line bundles, subject to the range condition on $j-i$.
\end{remark}

All our varieties $X$ in the present paper are smooth, and the
natural homomorphism  $K_0(X)\lra G_0(X)$ from the Grothendieck group of vector bundles on $X$ to that of the category $\Coh(X)$ of coherent sheaves is an isomorphism~\cite{Manin}. For this reason, we will write $K_0(X)$ for the Grothendieck group of either category.

$K_0(\PP^n)$ is a free abelian group on the symbols $[\mcO(i)]$ for $-n\le i \le 0$, see~\cite[Theorem 8.5]{Weibel} or~\cite{Manin}, for instance. For $-n\le i \le 0$ and $-m\le j\le 0$ the inequalities   $-m\le j-i\le n$ and decomposition \eqref{eq_f_decomposition} hold.
Applying Theorem~\ref{thm_pushforward}, we obtain:
\begin{prop}\label{prop_iso_noneq}
    There is an isomorphism of commutative associative rings
    \begin{equation}\label{eq_iso_wtR}
\wtR \cong \oplusop{n\ge 0} K_0(\PP^n)
    \end{equation}
    that sends $\wtK_{c,n}$ to $[\mcO_{\PP^n}(-c)]$.
\end{prop}
Under this isomorphism, the multiplication rule \eqref{eq_multiplication_cl}  lifts to an isomorphism \eqref{eq_f_decomposition} of vector bundles on $\PP^{n+m}$, which is a stronger statement than just a match of $\wtR$ with the Grothendieck ring.
This proposition is stated as Theorem~\ref{thm_isom_noneq} in the Introduction. To get a match between the formulas, replace $c=-i$ and $b=-j$ in \eqref{eq_f_decomposition}. Parameter $k$ in \eqref{eq_f_decomposition} is replaced by $m-i-k$ in
\eqref{eq_multiplication_cl}. One can then check that the entries of the binomials match, up to the symmetry $\binom{a}{b}=\binom{a}{a-b}$.

\begin{remark}
    The relation between the ring $\wtR$ and coherent sheaves on the disjoint union of $\PP^n$'s was originally observed by looking at the first relation in Example~\ref{example_motivate} and noticing that it is categorified by line bundles $\mcO(c+i)$ on $\PP^1$, in the sense of having short exact sequences
    \[
    0 \lra \mcO(c) \lra \mcO(c+1)^2\lra \mcO(c+2) \lra 0.
    \]
    In the Grothendieck group $K_0(\PP^1)$ there is a relation
    \[
2[\mcO(c+1)]=[\mcO(c)]+[\mcO(c+2)]
    , \]
    matching the above relation from Example~\ref{example_motivate}. Next, for line bundles on $\PP^2$, there is an exact sequence
    \[
    0\lra \mcO(c)\lra \mcO(c+1)^3 \lra \mcO(c+2)^3 \lra \mcO(c+3) \lra 0,
    \]
    descending in the Grothendieck group $K_0(\PP^2)$ to
    \[
    [\mcO(c)]-3[\mcO(c+1)]+3[\mcO(c+2)]- [\mcO(c+3)] = 0 ,
    \]
    which matches the second relation in Example~\ref{example_motivate}. From here it is an easy guess to defining the multiplication via the pushforward of the map $f_{n,m}$, adding $\mathbb{G}_m$-equivariance to introduce $q$ on the Grothendieck group level, and to the results of the paper.
\end{remark}

\begin{remark}
Note that, as functors between abelian categories, $(f_{n,m})^{\ast}$ is left adjoint to $(f_{n,m})_{\ast} $, while the right adjoint is given by
\begin{equation}
(f_{n,m})^{!}(\mcF) \cong f_{n,m}^{\ast}(\mcF)\otimes \mcO_{\PP^n\times\PP^m}(m,n).
\end{equation}
In particular, the left and right adjoints for $(f_{n,m})_{\ast}$ are related by a twist by an invertible functor. Consequently, one obtains an infinite in both directions chain of adjoint functors
\[
  \dots, (f_{n,m})^{\ast},(f_{n,m})_{\ast},(f_{n,m})^!,\dots \]
  Via the isomorphism \eqref{eq_iso_wtR} and the isomorphism $K_0(\PP^n\times \PP^m)\cong K_0(\PP^n)\otimes K_0(\PP^m)$, the induced map $[f^{\ast}]: \wtR \lra \wtR \otimes_{\Z} \wtR$ (summing over all $n,m$) is given by
  \[
  \wtK_{c,n}\lra \sum_{k=0}^n \wtK_{c,k}\otimes \wtK_{c,n-k},  \ \forall c\in \Z.
  \]
\end{remark}
The internal tensor product of vector bundles (or coherent sheaves) on $\PP^n$ induces a ring structure on $K_0(\PP^n)$, for each $n$. For the ring $\wtR$, this translates into a commutative ring structure $\ast$ on each homogeneous component $\wtR_n$, with $\wtK_{c,n}\ast \wtK_{d,n}=\wtK_{c+d,n}$. Ring $\wtR_n$ is generated by $\wtK_{1,n}$, with the defining relation $\wtR_n\cong \Z[\wtK_{1,n}]/((\wtK_{1,n}-1)^{n+1})$.
One can extend it to a (nonunital idempotented) ring structure on $\wtR$ via $\wtK_{c,n}\ast \wtK_{d,m}=0$ if $n\not= m$.
Adams operations carry over from $K_0(\PP^n)$ to $\wtR_n$ and make the latter into a $\lambda$-ring, via $\psi^k(\wtK_{c,n})=\wtK_{kc,n}$.

The Euler characteristic of ext groups descends to a bilinear form on $K_0(\PP^n)$:
\[
\chi([\mcF_0],[\mcF_1]) = \sum_i (-1)^i \dim \Ext^i(\mcF_0,\mcF_1).
\]
This bilinear form carries over to a form on $\wtR_n$, where it satisfies
\begin{equation*}
(x,y) = (-1)^n(y,x\ast \wtK_{n+1,n}), \ \ x,y\in \wtR_n, \ \  (\wtK_{c,n}, \wtK_{d,n}) = \binom{n +c-d}{n}.
\end{equation*}

\begin{remark}
    Pick $k\in \Z$ and consider the functor
    \[
    \mcG_k:\Coh(\PP)\lra D^b_{\mathrm{f.d.}}(\kk),
    \]
    into the derived category of complexes of vector spaces with finite dimensional total cohomology groups,
    \[
    \mcG_k(\mcF)=R\mathrm{Hom}(\mcO_{\PP^n}(k),\mcF)=R\Gamma(\PP^n,\mcF(-k)), \ \ \mcF\in\Coh(\PP^n).
    \]
    Functor $\mcG_k$ is symmetric monoidal, due to \eqref{eq_pullback_O} and
    \[
    \begin{aligned}
R\mathrm{Hom}_{\PP^{n+m}}\left(\mathcal{O}(k), (f_{n,m})_*(\mcF \boxtimes \mcG)\right)
&\simeq R\mathrm{Hom}_{\PP^n \times \PP^m}\left(f_{n,m}^* \mcO_{\PP^{n+m}}(k), \mcF \boxtimes \mcG\right) \\
&\simeq R\mathrm{Hom}_{\PP^n}(\mathcal{O}_{\PP^n}(k), \mcF) \otimes^L R\mathrm{Hom}_{\PP^m}(\mcO_{\PP^m}(k), \mcG).
\end{aligned}
    \]
\end{remark}
On the Grothendieck rings, the induced map
\[
[\mcG_k] \ : \ \wtR \lra \Z
\]
is a homomorphism taking $\wtK_{c,n}$ to $(-1)^n \binom{c+k-1}{n}$.

\begin{example}\label{example_worp}
    Consider the quotient map
    \[
 f_n \ : \ (\PP^1)^n \lra \PP^n,
 \]
 which can also be obtained by iterating the maps $f_{i,1}$, for $i=1,\dots,n-1$.
Consider the bundle $(\mcO_{\PP^1}(1))^{\boxtimes n}$ on $(\PP^1)^n$, which is the external tensor power of the bundle $\mcO_{\PP^1}(1)$. One can show that
\begin{equation}
     (f_n)_{\ast}(\mcO_{\PP^1}(1)^{\boxtimes n}) \cong \bigoplus_{m=0}^{n-1} \left\langle \begin{matrix} n \\ m \end{matrix} \right\rangle\mcO_{\PP^n}(m + 2 - n),
     \end{equation}
where $\left\langle \begin{smallmatrix} n \\ m \end{smallmatrix}\right\rangle$ is the Eulerian number (the number of permutations in $S_n$ with exactly $m$ descents, see~\cite{Pet}).
A proof (omitted) is similar to the one for Theorem~\ref{thm_pushforward}, with the Chu-Vandermonde identity replaced by
 the Worpitzky identity~\cite{Pet}:
\begin{equation}\label{eq_worpit}
x^n = \sum_{m=0}^{n-1} \left\langle \begin{matrix} n \\ m \end{matrix} \right\rangle \binom{x+m}{n}.
\end{equation}
This translates into the following equation in $\wtR$:
\[
(\wtK_{-1,1})^n = \sum_{m=0}^{n-1}
\left\langle \begin{matrix} n \\ m \end{matrix} \right\rangle\wtK_{n-m-2,n}.
\]

\end{example}
\begin{remark} The bundles in Theorem~\ref{thm_pushforward} are closely related to the classical secant, or Schwarzenberger, bundles on symmetric powers of $\PP^1$.
Let \[ N=n+m,\qquad d=j-i. \] By the projection formula,
\[ f_{n,m*}\bigl(\mcO_{\PP^n}(i)\boxtimes \mcO_{\PP^m}(j)\bigr) \cong \mcO_{\PP^N}(i)\otimes f_{n,m*}\bigl(\mcO_{\PP^n}\boxtimes \mcO_{\PP^m}(d)\bigr).
\]
The direct-image realization of exterior powers of secant bundles (see, for example,~\cite[Theorem~3.1]{RaicuSam}) identifies the latter pushforward with
\[ f_{n,m*}\bigl(\mcO_{\PP^n}\boxtimes \mcO_{\PP^m}(d)\bigr) \cong \bigwedge^m \mathcal{E}_{N,\mcO_{\PP^1}(d+m-1)}, \]
where $\mathcal{E}_{N,\mcL}$ denotes the rank-$N$ secant bundle on $S^N(\PP^1)\cong\PP^N$ associated to a line bundle $\mcL$ on $\PP^1$. In the range \[ -m\le d\le n, \] one has
\[
\mathcal{E}_{N,\mcO_{\PP^1}(d+m-1)} \cong \mcO_{\PP^N}^{\oplus(d+m)} \oplus \mcO_{\PP^N}(-1)^{\oplus(n-d)}.
\]
Consequently,
\[
\begin{aligned} \bigwedge^m \mathcal{E}_{N,\mcO_{\PP^1}(d+m-1)} &\cong \bigoplus_{k=0}^m \binom{d+m}{k} \binom{n-d}{m-k} \mcO_{\PP^N}(k-m). \end{aligned}
\]
Tensoring with $\mcO_{\PP^N}(i)$ and substituting $d=j-i$ recovers precisely the decomposition \eqref{eq_f_decomposition} of Theorem~\ref{thm_pushforward}. \end{remark}

%
%

\section{\texorpdfstring{$\mathbb{G}_m$}{Gm}-equivariant coherent sheaves}
\label{sec_equiv_coherent}

Our goal in this section is to derive Theorem~\ref{thm_isom} in the Introduction. Let us denote $G=\mathbb{G}_m$, so that the subindex $m$ can be used elsewhere.
Consider the action of $G=\mathbb{G}_m$ on a 2-dimensional vector space by
\begin{equation}\label{eq_V_action}
V = \kk x \oplus \kk y, \qquad z \cdot x = zx, \qquad z \cdot y = z^{-1}y.
\end{equation}
Then $\ch_q(V) = q + q^{-1}.$
Identify
$Y_n = S^n(\PP^1) \cong \PP^n$
with the projective space of binary forms of degree $n$. Choose the standard equivariant linearization of $\mathcal{O}_{\mathbf{P}^n}(1)$, characterized by
\begin{equation}\label{eq_linearize}
H^0\bigl(\mathbf{P}^n, \mathcal{O}(1)\bigr) \cong \operatorname{Sym}^n(V)
\end{equation}
as $G$-representations. Its weights are $n, n-2, \ldots, -n$.
The multiplication map
\[
f_{n,m} \colon \PP^n \times \PP^m \longrightarrow \PP^{n+m}, \qquad ([P],[Q]) \longmapsto [PQ],
\]
is $G$-equivariant (in fact, it is $\mathrm{SL}_2$-equivariant), and
\begin{equation}\label{eq_eq_fnm}
f_{n,m}^* \mathcal{O}_{\PP^{n+m}}(1) \cong \mathcal{O}_{\PP^n \times \PP^m}(1,1)
\end{equation}
equivariantly.

\begin{remark}
    Let
\[
K_G(Y_n) := K_0\bigl(\operatorname{Coh}^G(Y_n)\bigr)
\]
be the Grothendieck group of $G$-equivariant coherent sheaves on $Y_n \cong \PP^n$. Since $Y_n$ is smooth, this group agrees with the Grothendieck group of $G$-equivariant vector bundles (the former is usually denoted $G_0$ and the latter $K_0$).

Let $q$ denote the class of the one-dimensional $G = \mathbb{G}_m$ representation of weight $1$. Then
\[
K_G(\operatorname{pt}) = R(G) \cong \mathbb{Z}[q, q^{-1}].
\]
We view $\mcO_{\PP^n}(1)$  as  a $\mathbb{G}_m$-equivariant bundle with the linearization determined by \eqref{eq_linearize}.
By the equivariant projective bundle theorem~\cite[Theorem 10]{Merkurjev} and Thomason~\cite{Thomason} (or~\cite[Section 5.7]{ChrissGinzburg}, over $\mathbb{C}$),
\[
K_G(\PP^n) \cong \bigoplus_{i=0}^{n} \mathbb{Z}[q, q^{-1}] \cdot [\mcO_{\PP^n}(-i)].
\]

Equivalently,
\[
[\mcO_{\PP^n}], [\mcO_{\PP^n}(-1)], \dots, [\mcO_{\PP^n}(-n)]
\]
form a basis of $K_G(\PP^n)$ over $\mathbb{Z}[q, q^{-1}]$. (Different linearizations of $\mcO(1)$ differ by tensoring with a character $q^a$. This multiplies $[\mcO(i)]$ by the unit $q^{ai}$, so the basis statement is independent of that normalization.)
\end{remark}

Applying the equivariant projective bundle theorem~\cite{Thomason} twice,
\[
K_G(\PP^n\times \PP^m) \cong K_G(\PP^n)\otimes_{\Z[q,q^{-1}]}K_G(\PP^m).
\]
 Equivariant Grothendieck group $K_G(\PP^n\times \PP^m)$ is a free $\Z[q,q^{-1}]$-module with a basis $\mcO(i,j)=\mcO_{\PP^n}(i)\boxtimes \mcO_{\PP^m}(j)$, over $0\le i\le n, 0\le j\le m$.

\begin{remark}
We use the standard notion of a $G$-linearized sheaf and equivariant line bundle as in \cite[Ch.~I, \S3, Def.~1.6]{MFK};
see also \cite[Def.~1.33 and Lem.~1.34]{Brion}.
Equivariant pullback is as in \cite[Tag 03LE, Lem.~39.12.2]{stacks-project}.
For equivariant Grothendieck groups and their functoriality we follow Thomason, as summarized in \cite[\S2.2]{Merkurjev}, and, over $\mathbb{C}$, as in Chriss-Ginzburg~\cite{ChrissGinzburg}.
\end{remark}

\begin{remark}
The Grothendieck group of the abelian category of equivariant coherent sheaves may be denoted $G_0^G(\PP^n)$,
whereas $K_0^G(\PP^n)$ denotes the Grothendieck group of equivariant vector bundles. Since $\PP^n$ is smooth, the natural map
\[
K_0^G(\PP^n) \longrightarrow G_0^G(\PP^n)
\]
is an isomorphism~\cite[Theorem 13]{Merkurjev}. We write $K_G(\PP^n)$ instead of $K^G_0(\PP^n)$, for short.
\end{remark}

With the above $\mathbb{G}_m$-action,
\[
\ch_q H^0\bigl(\PP^n, \mathcal{O}(t)\bigr) = \begin{bmatrix} n+t \\ n \end{bmatrix}, \qquad (t \geq 0).
\]
Indeed,
\[
\sum_{t \geq 0} \ch_q H^0\bigl(\PP^n, \mathcal{O}(t)\bigr) z^t = \prod_{a=0}^n \frac{1}{1 - z q^{n-2a}}.
\]

Finite, flat morphisms $f_{n,m}$ are $G$-equivariant and induce exact adjoint functors $(f_{n,m}^{\ast},(f_{n,m})_{\ast})$ between categories of $G$-equivariant coherent sheaves on $\PP^n\times \PP^m$ and $\PP^{n+m}$. (See~\cite[\S 2.2.1.]{Merkurjev} and~\cite[\S 1.5]{Thomason} for the exactness properties in the equivariant setting.)

\begin{thm}\label{thm_pushforward_equiv} Equivariant pushforward sheaf of $\mcO_{\PP^n}(i) \boxtimes \mcO_{\PP^m}(j)$ under the map $f=f_{n,m}$ splits as a direct sum of (equivariant) line bundles on $\PP^{n+m}$ if  $-m \le j - i \le n$. When this condition holds, the pushforward is given by:
\begin{equation}\label{eq_f_decomposition_equiv}
f_{\ast}(\mcO_{\PP^n}(i) \boxtimes \mcO_{\PP^m}(j)) \cong \bigoplus_{k=0}^{m} \bbinom{j-i+m}{k} \bbinom{i-j+n}{m-k}\mcO_{\PP^{n+m}}(i + k - m).
\end{equation}
\end{thm}
Our earlier notations are used here, with square brackets denoting balanced quantum binomials in \eqref{eq_balanced}. For a sheaf $\mcL$ and a Laurent polynomial $f(q)=\sum_{i\in \Z} a_i q^i\in \Z_+[q,q^{-1}]$, we write $f(q)\mcL$ to denote the direct sum $\oplus_i \mcL^{\oplus a_i} \langle i\rangle$.

\begin{proof}
The case $n=0$ or $m=0$ is trivial, so assume $n,m>0$ and put $N=n+m$.
Sheaf
\[
\mathcal{E}=f_*\bigl(\mcO_{\PP^n}(i)\boxtimes\mcO_{\PP^m}(j)\bigr).
\]
is a $G$-equivariant vector bundle,
since $f$ is finite and flat.

By the projection formula~\cite[Tag~01E6, Lemma~20.54.2]{stacks-project}, which is compatible with the $G$-linearizations, and the equivariant isomorphism \eqref{eq_eq_fnm}, we have
\[
\mathcal{E}(t)
\cong
f_{\ast}\bigl(\mcO(i+t)\boxtimes\mcO(j+t)\bigr)
\]
as $G$-equivariant sheaves.  Since $f$ is finite, its higher direct images vanish.  Thus, for every $t\in\mathbb Z$, there is an isomorphism of $G$-representations
\begin{equation*}
H^r\bigl(\PP^N,\mathcal E(t)\bigr)
\cong H^r\bigl(\PP^n\times\PP^m,
\mcO(i+t)\boxtimes\mcO(j+t)\bigr)\
\cong \bigoplus_{p+s=r}
H^p\bigl(\PP^n,\mcO(i+t)\bigr)\otimes
H^s\bigl(\PP^m,\mcO(j+t)\bigr),
\end{equation*}
where the second isomorphism is the equivariant K\"unneth isomorphism.

As in the non-equivariant proof (of Theorem~\ref{thm_pushforward} above), a line bundle on $\PP^n$ can have cohomology only in degree $0$ or $n$. Thus, the only possible nonzero intermediate cohomology occurs when one factor contributes top cohomology and the other contributes $H^0$.  The two possibilities are
\[
i+t\le -n-1,\ \  j+t\ge 0 \ \leftrightarrow \ j-i\ge n+1,
\]
and
\[
j+t\le -m-1,\ \  i+t\ge 0, \ \leftrightarrow \ i-j\ge m+1.
\]
Hence, if $-m\le j-i\le n$,
then
\[
H^r\bigl(\PP^N,\mathcal E(t)\bigr)=0,
\qquad 0<r<N,
\]
for every $t\in\mathbb Z$.  Horrocks' criterion therefore implies that the underlying vector bundle of $\mathcal E$ splits as a direct sum of line bundles.

We next note that this splitting may be chosen $G$-equivariantly. Write the distinct degrees occurring in the nonequivariant splitting in decreasing order.
The filtration obtained by successively collecting the summands of largest degree is the Harder--Narasimhan filtration of $\mathcal{E}$
with respect to $\mcO_{\PP^N}(1)$;
see \cite[Theorem~1.6.7 and Definition--Corollary~1.6.9]{HuLe}. In particular, this filtration is intrinsic.
Since the $G$-action preserves
$\mcO_{\PP^N}(1)$, uniqueness of the Harder--Narasimhan filtration
implies that this filtration is $G$-invariant.

Its successive quotients have underlying bundles of the form
$Q_a=\mcO_{\PP^N}(a)^{\oplus e_a}$.  Then
$Q_a\otimes\mcO_{\PP^N}(-a)$ is a $G$-equivariant vector bundle whose
underlying bundle is trivial.  Since
$H^0(\PP^N,\mcO_{\PP^N})=\Bbbk$, a $G$-linearization of a trivial
bundle is determined by a finite-dimensional representation of $G$.
Consequently,
\begin{equation}\label{eq_Qa}
Q_a\cong \mcO_{\PP^N}(a)\otimes M_a
\end{equation}
for some finite-dimensional $G$-representation $M_a$.
Since $N\ge 2$,
\[
\operatorname{Ext}^1_{\PP^N}\bigl(\mcO(a),\mcO(b)\bigr)
\cong H^1\bigl(\PP^N,\mcO(b-a)\bigr)=0
\]
for all $a,b$, and
\[
\operatorname{Ext}^1_{\PP^N}
\bigl(\mcO(a)\otimes M_a,\mcO(b)\otimes M_b\bigr)=0
\]
for all finite-dimensional $G$-representations $(M_a,M_b)$.
We claim that the successive extensions in the Harder--Narasimhan filtration split $G$-equivariantly.  Indeed, consider a $G$-equivariant short exact sequence
\[
0\lra A\lra B\stackrel{p}{\lra} Q_a\longrightarrow 0
\]
arising from this filtration, and suppose inductively that $A$ is a direct sum of bundles of the form
$\mcO(b)\otimes M_b$.
The vanishing above gives $\operatorname{Ext}^1_{\PP^N}(Q_a,A)=0$, and hence the natural map
\[
\operatorname{Hom}_{\PP^N}(Q_a,B)\longrightarrow
\operatorname{End}_{\PP^N}(Q_a)
\]
is surjective.  This is a homomorphism of $G$-representations.  Since
$G=\mathbb G_m$ is linearly reductive, taking $G$-invariants is exact, so the induced map
\[
\operatorname{Hom}_{\PP^N}(Q_a,B)^G\longrightarrow
\operatorname{End}_{\PP^N}(Q_a)^G
\]
is also surjective.  The identity map $\id_{Q_a}$ is $G$-invariant, and therefore lifts to a $G$-equivariant section $s:Q_a\lra B$ of $p$.  Thus the sequence splits $G$-equivariantly.  Applying this inductively to the Harder--Narasimhan filtration gives
\begin{equation}\label{eq_E_decomp}
\mathcal E\cong
\bigoplus_a \mcO_{\PP^N}(a)\otimes M_a
\end{equation}
as a $G$-equivariant vector bundle.

It remains to determine the characters of the multiplicity representations $M_a$.  For
$t\ge\max(-i,-j)=-\min(i,j)$, the chosen symmetric linearizations give
\begin{align*}
\ch_q H^0\bigl(\PP^N,\mathcal E(t)\bigr)
&=\ch_q H^0\bigl(\PP^n\times\PP^m,
\mcO(i+t)\boxtimes\mcO(j+t)\bigr)\\
&=\bbinom{t+i+n}{n}\bbinom{t+j+m}{m}.
\end{align*}
Set
\[
x=t+i+n,\qquad d=j-i+m-n.
\]
The splitting condition $-m\le j-i\le n$ is equivalent to $-n\le d\le m$.  The balanced quantum Chu--Vandermonde identity, in the form parallel to \eqref{eq_CVa}, is
\[
\bbinom{x}{n}\bbinom{x+d}{m}
=
\sum_{k=0}^{m}
\bbinom{n+d}{k}
\bbinom{m-d}{m-k}
\bbinom{x+k}{n+m}.
\]
Substituting  $x$ and $d$ above gives
\begin{equation}
\ch_q H^0\bigl(\PP^N,\mathcal E(t)\bigr)
=
\sum_{k=0}^{m}
\bbinom{j-i+m}{k}
\bbinom{i-j+n}{m-k}
\cdot
\bbinom{t+(i+k-m)+N}{N}.
\end{equation}
The last quantum binomial is the character of
$H^0\bigl(\PP^N,\mcO(i+k-m+t)\bigr)$ whenever the twist is nonnegative, and it is zero when the twist lies between $-N$ and $-1$.  For every $k$ for which the coefficient in the preceding sum is nonzero and every $t\ge -\min(i,j)$, these are the only two possibilities.  Thus the right-hand side is the character of the global sections of the equivariantly split bundle
\[
\mathcal E'
:=
\bigoplus_{k=0}^{m}
\bbinom{j-i+m}{k}
\bbinom{i-j+n}{m-k}
\mcO_{\PP^N}(i+k-m).
\]
Consequently,
\[
\ch_q H^0\bigl(\PP^N,\mathcal E(t)\bigr)
=
\ch_q H^0\bigl(\PP^N,\mathcal E'(t)\bigr)
\qquad
\text{for all }t\ge - \min(i,j).
\]
Finally, the character-valued Hilbert function determines an equivariantly split bundle in this situation.  Indeed, from the K\"unneth description above,
\[
H^0(\PP^N,\mathcal E(t))=0
\quad\text{for }t<-\min(i,j),
\]
and this space is nonzero for $t=-\min(i,j)$.  On the other hand, from the equivariant splitting in \eqref{eq_E_decomp}
we obtain
\[
H^0\bigl(\PP^N,\mathcal E(t)\bigr)
\cong
\bigoplus_a
H^0\bigl(\PP^N,\mcO(a+t)\bigr)\otimes M_a.
\]
Since $H^0(\PP^N,\mcO(d))=0$ precisely when $d<0$, if $a_{\max}$ denotes the largest twist occurring in the splitting, then
\[
H^0(\PP^N,\mathcal E(t))=0
\quad\text{for }t<-a_{\max},
\qquad
H^0(\PP^N,\mathcal E(-a_{\max}))\ne 0.
\]
Hence $a_{\max}=\min(i,j)$.

At $t=-\min(i,j)$ only the summands of this largest twist contribute to $H^0$, so their multiplicity character is determined.  Increasing $t$ by one determines, successively, the multiplicity character in each lower twist.  The same argument applies to $\mathcal E'$, and the equality of characters above shows that all multiplicity characters agree.  Since finite-dimensional $\mathbb G_m$-representations are determined by their characters, $\mathcal E\cong\mathcal E'$ equivariantly and decomposition \eqref{eq_f_decomposition_equiv} follows.
\end{proof}

Note also that relation \eqref{eq_q_binom} matches the corresponding exact sequences involving equivariant bundles $\mcO_{\PP^n}(-c-k)$.
The following corollary is stated as Theorem~\ref{thm_isom} in the introduction.
\begin{cor}\label{cor_isom}
    There is a natural isomorphism of graded $\Z[q,q^{-1}]$-algebras
    \begin{equation}\label{eq_thm_isom_n}
    R\ \cong \oplusop{n\ge 0}\, K_{\mathbb{G}_m}(\PP^n)
    \end{equation}
    that takes $K_{c,n}$ to the symbol of the equivariant line bundle $\mcO(-c)$ on $\PP^n$.
\end{cor}
Defining relations \eqref{eq_multiplication} in $R$ are categorified by the direct sum decompositions~\eqref{eq_f_decomposition_equiv}.

The $q$-semilinear Euler characteristic form
\begin{equation}\label{eq_Euler_form}
K_G(\PP^n) \times K_G(\PP^n) \lra \Z[q,q^{-1}], \ \
([\mcL_1],[\mcL_2]) = \sum_r (-1)^r \ch_q(\Ext^r(\mcL_1,\mcL_2)),
\end{equation}
is given on the line bundles with the standard equivariant structure by
\[
([\mcO(-c)],[\mcO(-d)]) = \bbinom{n+c-d}{n}, \ \ c,d\in \Z,
\]
where one extends the balanced quantum binomials to arbitrary integral upper entries by
\[
\bbinom{a}{n}=\frac{[a][a-1]\dots [a-n+1]}{[n]!}.
\]
This induces a $q$-semilinear form on $R$, with
\[
(K_{c,n},K_{d,m}) = \delta_{n,m}\bbinom{n+c-d}{n}, \ \ c,d\in \Z.
\]
Under this semilinear form, the multiplication on $R$, induced, via the isomorphism \eqref{eq_thm_isom_n}, by the maps $[(f_{n,m})_{\ast}]$, over all $n,m$, is adjoint to the comultiplication, induced by the maps $[(f_{n,m})^{\ast}]$, over all $n,m$:
\[
(z,xy )_{n+m}=(\Delta(z),x\otimes y)_{n,m},
\]
and
\[
\Delta(K_{c,n})=
\sum_{k=0}^n K_{c,k}\otimes K_{c,n-k}.
\]

In the formulas \eqref{eq_K_and_inverse} the RHS expressions correspond to the elements $K_{1,1}-q^{\pm 1} K_{0,1}\in R_1$. Up to the minus sign, these match the images of skyscraper sheaves at $0$ and $\infty$, twisted by the one-dimensional representations $\langle 1\rangle$ and $\langle -1\rangle$ of $\mathbb{G}_m$, in the Grothendieck group. These two fixed point classes diagonalize the equivariant Euler form \eqref{eq_Euler_form}, after localization~\cite[\S 2.3]{Okounkov}.

\begin{example} Example~\ref{example_worp} has a $q$-analogue (or a  $\mathbb G_m$-equivariant analogue).  Set
\[
\left\langle \begin{matrix} n \\ m \end{matrix}\right\rangle_q
:=
\sum_{\substack{\sigma\in S_n\\ \mathrm{des}(\sigma)=m}}
q^{\,2\mathrm{maj}(\sigma)-nm}
\]
to be the balanced $q$-Eulerian number.  Equivalently, if $A_{n,m}(Q)$ denotes
Carlitz's $q$-Eulerian number~\cite{Carlitz}, then
\[
\left\langle \begin{matrix} n \\ m \end{matrix}\right\rangle_q
=q^{-nm}A_{n,m}(q^2).
\]
In these balanced conventions, Carlitz's $q$-Worpitzky identity takes the form
\begin{equation}\label{eq_worpit_q}
[x]^n
=
\sum_{m=0}^{n-1}
\left\langle \begin{matrix} n \\ m \end{matrix}\right\rangle_q
\bbinom{x+m}{n};
\end{equation}
see~\cite{Carlitz} (also~\cite{Carlitz2} and \cite[Chapter 6.5]{Pet}).  On the coherent sheaves side, the corresponding formula is an isomorphism
\[
(f_n)_*\bigl(\mcO_{\PP^1}(1)^{\boxtimes n}\bigr)
\cong
\bigoplus_{m=0}^{n-1}
\left\langle \begin{matrix} n \\ m \end{matrix}\right\rangle_q
\mcO_{\PP^n}(m+2-n)
\]
of $\mathbb G_m$-equivariant bundles, where a Laurent polynomial coefficient
is interpreted as the corresponding sum of character twists.  In particular,
this specializes at $q=1$ to the ordinary Eulerian decomposition in Example~\ref{example_worp}. In the ring $R$, we get
\[
(K_{-1,1})^n = \sum_{m=0}^{n-1}\left\langle \begin{matrix} n \\ m \end{matrix}\right\rangle_q K_{n-m-2,n}.
\]
\end{example}

\begin{remark} \label{rmk_sl2}
Consider the involution $\tau$ of $\Z[q,q^{-1}]$ given by  $\tau(q)=q^{-1}$ and denote by $\Z^{\tau}:=\Z[q,q^{-1}]^{\tau}$ the $\tau$-invariant subring of $\Z[q,q^{-1}]$. We have $\Z^{\tau}= \Z[q+q^{-1}]\subset \Z[q,q^{-1}]$.
Extend involution $\tau$ to $R$ by $\tau(K_{c,n})=K_{c,n}$ for all $n,c$. Defining relations \eqref{eq_q_binom} and \eqref{eq_multiplication} are invariant under $\tau$, since balanced quantum binomials are $\tau$-invariant. Coefficients in these relations lie in $\Z^{\tau}$.
The subring $R^{\tau}$ of $\tau$-invariant elements of $R$ is a free $\Z^{\tau}$-module with a basis of elements $K_{\mbs}$, over all $\mbs\in\mcS$.

\vspace{0.07in}

Suppose that the ground field $\kk$ has characteristic zero.
There is an $\mathrm{SL}_2$-equivariant version of Theorem~\ref{thm_isom}.
Let
\[
T= \left\{ \begin{pmatrix} z&0\\ 0&z^{-1} \end{pmatrix} :z\in\mathbb G_m \right\} \subset \mathrm{SL}_2 \]
be the diagonal maximal torus. Identifying $R(T)\cong\Z[q,q^{-1}]$, the nontrivial element of the Weyl group acts by $q\mapsto q^{-1}$.
Restriction of representations to $T$ therefore identifies
$ R(\mathrm{SL}_2) \cong R(T)^W \cong \Z^\tau. $
Thus $R^\tau$ may naturally be viewed as an $R(\mathrm{SL}_2)$-algebra.
The natural action of $\mathrm{SL}_2$ on $V\cong\kk^2$ induces actions on $ \PP^n\cong \PP(\operatorname{Sym}^n V)$,
and the multiplication maps $ f_{n,m}$ are $\mathrm{SL}_2$-equivariant. The equivariant projective bundle theorem gives
\[
K_{\mathrm{SL}_2}(\PP^n) \cong \bigoplus_{i=0}^n \, R(\mathrm{SL}_2)[\mcO_{\PP^n}(-i)] \cong \bigoplus_{i=0}^n \, \Z^{\tau}[\mcO_{\PP^n}(-i)].
\]
The proof of Theorem~\ref{thm_pushforward_equiv} carries over to $\mathrm{SL}_2$ and  the equivariant splitting argument applies. Indeed, the Harder--Narasimhan filtration is $\mathrm{SL}_2$-invariant by uniqueness, and, in characteristic zero, $\mathrm{SL}_2$
is linearly reductive, so taking invariants is exact and the successive extensions split equivariantly.
The resulting multiplicity representations restrict to $T$-representations whose characters are exactly the coefficients in \eqref{eq_f_decomposition_equiv}.
In particular, these coefficients are $\tau$-invariant and belong to $\Z^{\tau}\cong R(\mathrm{SL}_2).$
Consequently, there is a natural isomorphism of graded $\Z^\tau$-algebras
\begin{equation} \label{eq_iso_SL2}
R^\tau \ \cong\ \bigoplus_{n\geq0}K_{\mathrm{SL}_2}(\PP^n), \qquad K_{c,n}\longmapsto[\mcO_{\PP^n}(-c)],
\end{equation}
where the multiplication on the right-hand side is induced by the $\mathrm{SL}_2$-equivariant pushforward functors $(f_{n,m})_*$. \end{remark}

%
%

\section{Proof of Proposition~\ref{prop_binom}}
\label{subsec_prop_binom}

 Let $E$ be the shift operator acting on functions of $c$, defined by $(E f)(c) = f(c+1)$. It $q$-commutes with the invertible scaling operator $(Sf)(c)=q^c f(c)$, so that $ES=qSE$. Let us write $F_{c,n}=\bmat{K}{c}{n}$. Note that $E (F_{c,n})=F_{c+1,n}$, viewing $E$ as acting on rational functions in $q^c$ with coefficients in $\Q(q,K)$.  We can rewrite the summation on the LHS of \eqref{eq_q_binom_2} as a polynomial operator in $E$ acting on $F_{c,n}$:
 $$\left( \sum_{k=0}^{n+1} (-1)^k \bbinom{n+1}{k} E^k \right) F_{c,n}.$$
 Via the $q$-binomial theorem (for balanced $q$-binomials), this operator factors as the product
 \begin{equation}\label{eq_E}
  \prod_{i=0}^n (1 - q^{n-2i} E).
 \end{equation}
 Furthermore, write $F_{c,1}$ via the operator $S$ as follows
 \[
 F_{c,1}= \left( \frac{K S - K^{-1} S^{-1}}{q - q^{-1}} \right) (1),
 \]
 and, likewise, let
 \[\mathbf{F}_n := \frac{1}{[n]!} \prod_{m=0}^{n-1} \frac{K q^{-m} S - K^{-1} q^m S^{-1}}{q - q^{-1}}, \ \ F_{c,n}= \mathbf{F}_n(1).
 \]

 Since $S$ commutes with the coefficients $q^{\pm 1},K^{\pm 1}$, we can expand the product for $\mathbf{F}_n$ into a linear combination of the operators $S^{n-2j}$ for $j=0, 1, \dots, n$:$$\mathbf{F}_n = \sum_{j=0}^n C_j S^{n-2j},$$
 with coefficients $C_j \in \mathbb{Q}(q,K)$.
 Evaluating this operator on the constant function $1$ expresses $F_{c,n}$ as a linear combination of the eigenfunctions $S^{n-2j}(1) = q^{(n-2j)c}$.
 Note that each function $S^m(1)(c)=q^{mc}$ is an eigenfunction of $E$ with eigenvalue $q^m$:
 $$E(S^{m}(1)) = q^{m} S^{m}(1).$$
 Now apply the factored operator \eqref{eq_E} to an arbitrary  component $S^{n-2j}(1)$ of $F_{c,n}$:
 $$\left( \prod_{i=0}^n (1 - q^{n-2i} E) \right) S^{n-2j}(1) = \left( \prod_{i=0}^n (1 - q^{n-2i} q^{n-2j}) \right) S^{n-2j}(1)=0.$$
 For any $j \in \{0, 1, \dots, n\}$, the factor  at $i = n-j$ yields $0$.
 Since the operator \eqref{eq_E} annihilates every individual basis component $S^{n-2j}(1)$ in the expansion of $\mathbf{F}_n(1)$, Proposition~\ref{prop_binom} follows.  $\square$

\printbibliography

\vspace{2em} \noindent
Department of Mathematics, Johns Hopkins University, Baltimore, MD 21218, USA\\
\textit{E-mail address:} \href{mailto:khovanov@jhu.edu}{khovanov@jhu.edu}

\end{document}